\documentclass{amsart}
\usepackage{amsfonts}
\usepackage{amsmath,amsthm}
\usepackage{amsfonts,amssymb}
\usepackage{CJK}
\usepackage{enumerate}

\newtheorem{thm}{Theorem}[section]
\newtheorem{cor}[thm]{Corollary}
\newtheorem{lem}[thm]{Lemma}

\theoremstyle{remark}
\newtheorem{rem}[thm]{Remark}

\numberwithin{equation}{section}

\newcommand{\al}{\alpha}

\def \b{\beta}

\def\oz{\omega}
\def\lz{\lambda}
\def\Lz{\Lambda}

\def\az{\alpha}

\def\({\Bigl(}
\def \){ \Bigr)}

 \def\az{{\alpha}}

 \def\lz{{\lambda}}
 
 \def\oz{{\omega}}

 \def\RR{{\mathbb R}}

\def\a{{\mathbf a}}
 \def\b{{\mathbf b}}

\def\Lz{\Lambda}

\def\Lz{\Lambda}

\begin{document}
\def\RR{\mathbb{R}}
\def\Exp{\text{Exp}}
\def\FF{\mathcal{F}_\al}

\title[] {Weighted $L_p$ Markov factors with doubling weights on the ball}

\author[]{Jiansong Li} \address{ School of Mathematical Sciences, Capital Normal
University, Beijing 100048,
 China}
\email{2210501007@cnu.edu.cn}

\author[]{ Heping Wang} \address{ School of Mathematical Sciences, Capital Normal
University, Beijing 100048,
 China}
\email{ wanghp@cnu.edu.cn}

\author[]{ Kai Wang} \address{School of Mathematics and Information Sciences, Langfang Teachers University, Langfang 065000, China}
\email{cnuwangk@163.com}

%\date{\today}


\begin{abstract} Let $L_{p,w},\ 1 \le p<\infty,$ denote the weighted $L_p$ space
of functions on the unit ball $\Bbb B^d$ with  a doubling weight
$w$ on $\Bbb B^d$. The Markov factor for $L_{p,w}$  on a
polynomial $P$ is defined by $\frac{\|\, |\nabla
P|\,\|_{p,w}}{\|P\|_{p,w}}$, where $\nabla P$ is the gradient of
$P$. We investigate the worst case Markov factors for $L_{p,w}\
(1\le p<\infty)$ and obtain that the degree of these factors are
at most $2$. In particular, for the Jacobi weight
$w_\mu(x)=(1-|x|^2)^{\mu-1/2}, \ \mu\ge0$, the exponent $2$ is
sharp. We also study the average case  Markov factor for $L_{2,w}$
on random polynomials with independent $N(0, \sigma^2)$
coefficients and obtain that the upper bound of the average
(expected) Markov factor is order degree to the $3/2$, as compared
to the degree squared worst case upper bound.
\end{abstract}

\maketitle
\input amssym.def

\section{Introduction }

\ Let $\mathbb{B}^{d}=\{x\in \mathbb{R}^d:\,|x|\leq 1\}$ denote
the unit ball of $\mathbb{R}^d$ and $\mathbb{S}^{d}=\{x\in
\mathbb{R}^{d+1}:\, |x|= 1\}$ denote the unit sphere of
$\mathbb{R}^{d+1}$, where $x\cdot y$ is the usual inner product
and $|x|=(x\cdot x)^{1/2}$ is the usual Euclidean norm.  Given a weight function $w$, we
denote by $L_{p,w}:= L_{p}(\Bbb B^d, w(x)dx)\ (1\le p\le\infty)$
the Lesbegue space on $\Bbb B^d$ endowed with the norm
$$\|f\|_{p,w}:=\left\{
\begin{aligned}&\Big(\int_{\mathbb{B}^{d}}|f(x)|^pw(x)dx\Big)^{1/p},\ \ &1\leq
p<\infty,
\\
&{\rm ess}\sup\limits_{x\in\mathbb{B}^d}|f(x)|,\ \ \
&p=\infty,\end{aligned}\right.$$ and we set, for a measurable
subset $\textbf{E}$ of $\Bbb B^d$,
$$|\textbf E|:=\int_{\textbf{E}}dx \ \ {\rm and}\ \ w(\textbf{E}):=\int_{\textbf{E}}w(x)dx.$$
For simplicity, we write $L_p$ and $\|\cdot\|_p$ for $L_{p,w}$ and
$\|\cdot\|_{p,w}$ when $w(x)=1$.

We introduce a distance on $\Bbb B^d$:
$$d(x,y):=\arccos\{x\cdot y+ \sqrt{1-|x|^2}\sqrt{1-|y|^2}\},\ \ x,y\in\mathbb{B}^d,$$
and write
$$ \mathbf{B}(x,r):=\{y\in \Bbb B^d:\, d(x,y)\le r\}.$$
It also can be seen that the distance $d$ is
equivalent to the distance $\tilde{d}$ on $\Bbb B^d$, which is
defined by
$$\tilde{d}(x,y):=\sqrt{|x-y|^2+(\sqrt{1-|x|^2}-\sqrt{1-|y|^2})^2},\
\ x,y\in\mathbb{B}^d.$$We also  define the
geodesic distance on $\mathbb{S}^{d}$ by$$\rho(x,y):=
\arccos(x\cdot y),\ \ x,y\in\mathbb{S}^{d}.$$

A weight function $w$ is a doubling weight if there exists a constant
$L>0$, depending only on $d$ and $w$, such that
$$w(\mathbf{B}(x,2r))\leq L\, w(\mathbf{B}(x,r)),\ {\rm\ for\ all}\ x\in\mathbb{B}^{d},\ r>0\ ,$$
where the least constant $L_w$ is called the doubling constant of
$w$. In particular, the Jacobi weight
$w_\mu(x)=(1-|x|^2)^{\mu-1/2}\ (\mu \geq 0)$ is the doubling
weight on $\mathbb{B}^{d}$. This follows from the fact  (see
\cite[Lemma 5.3]{PX}) that for $r>0$, $x\in \Bbb B^d$,
\begin{equation}\label{1.1}
  w_\mu(\textbf{B}(x,r)):=\int_{\textbf{B}(x,r)}w_\mu(y)dy\asymp r^d(r+\sqrt{1-|x|^2})^{2\mu}.
\end{equation}Here we use the notion $A_n\asymp B_n$ to express
$A_n\ll B_n$ and $A_n\gg B_n$, and $A_n\ll B_n$ ($A_n\gg B_n$)
means that there exists that a constant $c>0$ independent of $n$
such that $A_n\leq cB_n$ ($A_n\geq cB_n$).
According to \eqref{1.1}, we also have for a positive integer $n$,
\begin{equation}\label{1.2}
  w_\mu(\textbf{B}(x,1/n))\asymp n^{-d}\mathcal{W}_\mu(n;x):=n^{-d}(\sqrt{1-|x|^2}+1/n)^{2\mu}.
\end{equation}

Let $\Pi_n^{d}$ be the space of all polynomials in $d$ variables
of total degrees at most $n$. For $\Omega\subset \Bbb R^d$, we use
$\Pi_n(\Omega)$ to denote the restriction to $\Omega$ of $\Pi_n^d$.

In this paper we study  Markov factors. Suppose that $X$ is a
Banach space consisting of  functions defined on a domain of $\Bbb
R^d$ with norm $\|\cdot\|_X$. The Markov factor for $X$ on a
polynomial $P$ is defined by
$$\frac{\|\, |\nabla P|\,\|_{X}}{\|P\|_X},$$ where
$\nabla=(\partial_1,\dots,
\partial_d):=(\frac{\partial }{\partial x_1},\dots, \frac{\partial
}{\partial x_d})$ is the gradient operator. The worst case Markov
factor for $X$  is defined by
$$M_n^{\rm wor}(X):=\sup_{0\neq P\in\Pi_n^{d}}\frac{\|\,|\nabla P|\,\|_X}{\|P\|_X}.$$

If  $X$ is a Hilbert space, then we can discuss the average case
Markov factor first introduced  in \cite{B}. Let $\mathcal{B}_n
:=\{P_1,P_2,\dots,P_{N_n^d}\}$ be an orthonormal basis for
$\Pi_n^d$ in  $X$, where
 $N_n^d$ denotes the dimension of $\Pi_n^d$. A random polynomial in $\Pi_n^d$ is defined by
$$P_{\a}(x)=\sum_{j=1}^{N_n^d}a_jP_j(x),$$ where ${\a}=(a_1,a_2,\dots,a_{N_n^d})^T\in \Bbb R^{N_n^d}$, the coefficients
$a_j$ are independent $N(0,\sigma^2)$ random variables with the
common normal density functions $$\frac
1{\sqrt{2\pi}\sigma}e^{-\frac {a_j^2}{2\sigma^2}},\ \ 1\le j\le
N_n^d.$$

The average case Markov factor for $X$ is defined by
\begin{align*}
M_n^{\rm ave}(X)&:=\Bbb E\Big(\frac{\|\,|\nabla P_{\bf
a}|\,\|_{X}}{\|P_{\bf a}\|_{X}}\Big)\\ &:=\frac1{(2\pi
\sigma^2)^{N_n^d/2}}\int_{\Bbb R^{N_n^d}}\frac{\| \, |\nabla
P_{\a}|\,\|_{X}}{\|P_{\a}\|_{X}}\, {e^{-\frac{\|{\bf
a}\|_2^2}{2\sigma^2}}}d{\a},
\end{align*}
where $$\|\a\|_2^2:=\sum_{j=1}^{N_n^d}|a_j|^2,\ \ \ d{\a}:=da_1da_2\dots da_{N_n^d}.$$

This paper is devoted to investigating the worst case Markov
factors for $L_{p,w},\ 1\le p<\infty$ and  the average  case
Markov factors for $L_{2,w}$, where $w$ is a doubling weight on
$\Bbb B^d$. We recall previous results about the worst and average
case Markov factors.

For $X = L_{p}([-1,1]),\ 1\le p\le\infty$, the classical Markov
inequality states that
\begin{equation}\label{1.3}
  M_n^{\rm wor}(L_{p}([-1,1])):=\sup_{0\neq P\in \Pi_n^1}\frac{\|P'\|_p}{\|P\|_p}\asymp n^2.
\end{equation}
In fact, \eqref{1.3} was proved by A. A. Markov for $p=\infty$ in
\cite{M} and by Hille, Szeg\"o and Tamarkin for $1\le p<\infty$ in
\cite{HST}.

For  $X = L_{p}({\bf K},dx),\ 1\le p\le \infty$ with convex body
${\bf K}$ in $\Bbb R^d$, Wilhelmsen obtained in \cite{Wi} for
$p=\infty$, and Kro\'o and S. R\'ev\'esz in \cite{KS} for $1\le
p<\infty$ that
$$M_n^{\rm wor}(L_p({\bf K},dx))\ll n^2.$$

Suppose that $1\le p< \infty$,
$w_{\alpha,\beta}(x)=(1-x)^\alpha(1+x)^\beta$ is the Jacobi weight
on $[-1,1]$, and  $\az,\beta>-1$. It is proved   that (see
\cite{G})
 \begin{equation}\label{1.4}
 M_n^{\rm wor}(L_{p}([-1,1],w_{\az,\beta}(x)dx)\asymp n^2,
 \end{equation}
and (see \cite{B}) \begin{equation*}M_n^{\rm
ave}(L_2([-1,1],w_{\alpha,\beta}(x)dx)) \asymp
n^{3/2}.\end{equation*} This means that  the worst and average
case Markov factors for $L_2([-1,1],w_{\alpha,\beta}(x)dx)$ are
different.

For $X = L_{p}([-1,1], w(x)dx)\,(1\le p<\infty)$ with doubling
weight $w$ on $[-1,1]$,  Mastroianni and Totik  proved in
\cite[Theorem 7.4]{MT} that
\begin{equation}
  M_n^{\rm wor}(L_{p}([-1,1], w(x)dx))\ll n^2.\label{1.5}
\end{equation} Wang, Ye and Zhai obtained in \cite{WYZ} that
\begin{equation}\label{1.6}
  n^{-1/2}M_n^{\rm wor}(L_2([-1,1],w(x)dx))\ll M_n^{\rm ave}(L_2([-1,1],w(x)dx))\ll n^{3/2}.
\end{equation}
In general, for a doubling weight $w$ on $[-1,1]$, the worst case
Markov factor for  $M_n^{\rm wor}(L_2([-1,1],w(x)dx))$ is $n^2$.
In this case, the average  case Markov factor $M_n^{\rm
ave}(L_2([-1,1],w(x)dx))$ is just $n^{3/2}$.

There are many other paper devoted to discussing the Benstein or
(and) Markov type inequalities, see for example, \cite{BE, D, DP,
Di, Du,  J, Kr1, Kr2, Kr3, Kr4, Sa}. We remark that there are few
results about the weighted $L_p$ Markov type inequalities on
compact domains of $\Bbb R^d\ (d\ge2)$  for $1\le p<\infty$.

 The main purpose of this paper is to extend \eqref{1.5} to the spaces $L_{p,w}$ and \eqref{1.6} to the spaces $L_{2,w}$, where $w$ is a doubling weight on $\Bbb B^d$.
 Our results can be formulated as follows.

\begin{thm}Let $d\ge2$, $1\leq p <\infty$, and $w$ be a doubling weight on $\Bbb B^d$. Then we have
$$
M_n^{\rm wor}(L_{p,w}):=\sup_{0\neq P\in\Pi_n^d}\frac{\|\,|\nabla P|\,\|_{p,w}}{\|P\|_{p,w}}\ll n^2.
$$In particular, for the Jacobi weight $w_\mu(x)=(1-|x|^2)^{\mu-1/2}\ (\mu \geq 0)$, the exponent $2$ is sharp, that is
$$
M_n^{\rm wor}(L_{p,w_\mu}):=\sup_{0\neq P\in\Pi_n^d}\frac{\|\,|\nabla P|\,\|_{p,w_\mu}}{\|P\|_{p,w_\mu}}\asymp n^2.
$$
\end{thm}

\begin{thm}Let $d\ge2$ and $w$ be a doubling weight on $\Bbb B^d$. Then we have
$$ M_n^{\rm ave}(L_{2,w}):=\Bbb E\Big(\frac{\|\,|\nabla P_{\bf
a}|\,\|_{2,w}}{\|P_{\bf a}\|_{2,w}}\Big)\ll
n^{3/2}.$$\end{thm}

We organize this paper as follows. In Section 2, we give the
integration representation of partial derivatives of polynomials
and provide a brief introduction about the average case Markov
factors. After that, in Section 3, we obtain the upper estimates
of the partial derivative of integral  kernels. Section 4 is
devoted to discussing  the Christoffel function on $\Bbb B^d$. In
Section 5, we give the proofs of the main results.

\section{Preliminaries}

\subsection{Integration representation  of  partial
derivatives of  polynomials }

\

We denote  by $\mathcal{V}_n^d$ the space of all  polynomials of degree $n$
which are orthogonal to lower degree polynomials in $L_{2,w_\mu}.$
It is well known (see \,\cite{DX}, p. 38 or\,p. 229) that the spaces
$\mathcal{V}_n^d$ are just the eigenspaces corresponding to the
eigenvalues $-n(n+2\mu+d-1)$ of the second-order differential
operator
$$D_\mu:=\triangle-(x\cdot\nabla)^2-(2\mu+d-1)x\cdot \nabla,$$
where the $\triangle$ is the Laplace operator. Also, the spaces
$\mathcal{V}_n^d$  are mutually orthogonal in $L_{2,w_\mu}$ and
$$\Pi_n^d= \ \bigoplus_{k=0}^{n}\mathcal{V}_k^d,\ \ \ \ \  L_{2,w_\mu}= \ \bigoplus_{n=0}^{\infty}\mathcal{V}_n^d. $$

 The orthogonal projector $$Proj_n :\, L_{2,w_\mu}\rightarrow \mathcal{V}_n^d$$ can be written
 as$$ (Proj_n f)(x)=\int_{\mathbb{B}^{d}}f(y)P_n(w_\mu; x, y)w_{\mu}(y)dy,$$ where $P_n(w_\mu; x, y)$
 is the reproducing kernel of $\mathcal{V}_n^d$. For  $\mu>0$, we
 have (see \cite{Xu1})
 $$P_n(w_\mu; x, y)= b_d^\mu b_1^{\mu-1/2}\frac{\lambda+n}{\lambda}\int_{-1}^{1}C_n^\lambda( x\cdot y
 + u\sqrt{1-|x|^2}\sqrt{1-|y|^2})(1-u^2)^{\mu-1}du,$$
where $C_n^\lz$ is the $n$-th degree Gegenbauer polynomial,
$$\lz=\mu+\frac{d-1}2\qquad {\rm and}\qquad
b_d^\mu:=\Big(\int_{\Bbb B^d}(1-|x|^2)^{\mu-1/2}dx\Big)^{-1}.$$ The
case $\mu=0$ is a limit case and we have
$$P_n(w_0; x, y)= b_d^0 \frac{\lambda+n}{2\lambda}[C_n^\lambda(x\cdot y+
\sqrt{1-|x|^2}\sqrt{1-|y|^2})+C_n^\lambda(x\cdot y-
\sqrt{1-|x|^2}\sqrt{1-|y|^2})].$$

Let $\eta$ be a
$C^\infty$-function on $[0,\infty)$ with $\eta(t)=1$ for $t\in
[0,1]$ and $\eta(t)=0$ for $t\ge2$. We define
\begin{equation}\label{2.1}L_n(w_\mu;x,y)=\sum_{j=0}^{\infty}\eta(\frac{j}n)P_j(w_\mu; x, y)=\sum_{j=0}^{2n-1}\eta(\frac jn)P_j(w_\mu; x, y).\end{equation}

 It is well known that $$L_n(w_\mu;x,y)=L_n(w_\mu;y,x),\ \ x,y\in\Bbb B^d,$$ $$ L_n(w_\mu;\cdot,y)\in\Pi_{2n}^d\ \ {\rm for\ fixed} \ y\in \Bbb B^d,$$ and for any $P\in\Pi_n^d$,
 \begin{equation}\label{2.2}P(x)=\int_{\Bbb B^d}P(y)L_{n}(w_\mu;x,y)w_\mu(y)dy,\ \ x\in
 \Bbb B^d,\end{equation}
 which deduces that
 \begin{equation}\label{2.3}\partial_i P(x)=
 \int_{\Bbb B^d}P(y)\partial_{x_i} L_n(w_\mu;x, y)w_\mu(y)dy,\ \ 1\le i\le d,  \end{equation}
where $\partial_{x_i} L_n(w_\mu;x, y)$ means the partial derivative about $x_i$.
So in order  to estimate the norm of  $\partial_i P(x)$, it is
vital to estimate $\partial_{x_i} L_n(w_\mu;x, y)$.

The following lemma will be used in the following sections.

\begin{lem}(See \cite[Lemma 4.6 ]{PX}.) Let $0<p<\infty,\ \mu\ge 0,$ and
$\sigma>\frac{d}{p}+2\mu|\frac{1}{p}-\frac{1}{2}|$.
 Then \begin{equation*}J_p:=\int_{\mathbb{B}^d}\frac{w_\mu(y)dy}{\mathcal{W_\mu}(n;y)^{\frac{p}{2}}(1+nd(x,y))^{\sigma p}}
 \ll n^{-d}{\mathcal{W_\mu}(n;x)}^{1-\frac{p}{2}}.\end{equation*}\end{lem}

\subsection{Average case Markov factors}

\

Let $\mathcal{B}_n :=\{P_1,P_2,\dots,P_{N_n^{d}}\}$ be an
orthonormal basis for $\Pi_n^d$ with respect to the Hilbert space
$X=L_{2,w}$, where $w$ is a doubling weight on $\Bbb B^d$. We also
suppose that the basis $\mathcal{B}_n$ is  consistent with the
degree, i.e., for $i < j,\ \deg(P_i)\leq\ \deg(P_j)$.
Then
for a random polynomial
$P_{\a}(x)=\sum\limits_{j=1}^{N_n^d}a_jP_j(x)$, we have
$$\partial_i P_{\a}\in \Pi_{n-1}^d,\ 1\le i\le d,$$where
$\a=(a_1,a_2,\dots,a_{N_n^d})^T\in \Bbb R^{N_{n}^d}$, and the
coefficients $a_i$ are independent $N(0,\sigma^2)$ random
variables. This means that
 $$\partial_i P_{\a}(x)=\sum_{j=1}^{N_{n-1}^d}b_j^iP_j(x),\ 1\le i\le d.$$ Here $\b^i=(b_1^i,b_2^i,\dots,b_{N_{n-1}^d}^i)^T\in \Bbb R^{N_{n-1}^d}$
 with $\b^i =\mathbb{D}_i\a,$  and the matrix $\mathbb{D}_i\in \Bbb R^{{N_{n-1}^d}\times{N_n^d}}$ is called the corresponding differentiation matrix.
Using this notation we may write $$\frac{\|\partial_i
P_\a\|_{2,w}}{\|P_\a\|_{2,w}}=\frac{\|\mathbb{D}_i
\a\|_2}{\|\a\|_2},\ \ 1\le i\le d.$$
It follows from \cite[Proposition 2.5 and Lemma 2.6]{B} that for $X=L_{2,w}$,
\begin{align}\label{2.4}M_n^{\rm ave}(X)&=\Bbb E\Big(\frac{\|\,|\nabla P_{\bf a}|\,\|_{2,w}}{\|P_\a\|_{2,w}}\Big)=\Bbb E \Big(\sum\limits_{i=1}^{d}\frac{\|\Bbb
D_i\a\|_2^2}{\|\a\|_2^2}\Big)^{\frac{1}{2}}\notag\\&\asymp
\sum\limits_{i=1}^{d}\Bbb E\frac{\|\Bbb D_i\a\|_2}{\|\a\|_2}\asymp
n^{-\frac{d}{2}}\sum\limits_{i=1}^d\sqrt{{\rm tr}({\Bbb D}_i^T
{\Bbb D}_i)}\ ,
\end{align}
where ${{\Bbb D}_i}^T$ is the transpose of ${\Bbb D}_i$, ${\rm tr}
(Q)$ is the trace of a matrix $Q$, and
\begin{equation}\label{2.5}
{\rm tr}({\Bbb D}_i^T {\Bbb D}_i)=\sum_{j=1}^{N_n^d}
\|\partial_iP_j\|_{2,w}^2,\ \ 1\le i\le d.
\end{equation}

\section{Upper estimates of partial derivatives of $L_n(w_\mu;x, y)$ }

In this section, we will estimate the upper bound of
$\partial_{x_i} L_n(w_\mu;x,y)$. The following lemma states that
$L_n(w_\mu;x, y)$ is a Lip 1 function in $x$  with respected to
the distance $d$ on $\mathbb{B}^d$.

\begin{lem}(See \cite[Proposition 4.7 ]{PX}.) Let $y,u\in \mathbb{B}^d,\ n\in\mathbb{N}_+$, $\delta>0$, and $\mu\ge 0.$ Then for all
$ z,\, \xi\in \mathbf{B}(y,\frac{\delta}{n}),$ and any $k>0$, we
have
\begin{equation}|L_n(w_\mu;y,u)-L_n(w_\mu;z,u)|\ll
\frac{n^{d+1}d(y,z)}{\sqrt{\mathcal{W_\mu}(n;u)}\sqrt{\mathcal{W_\mu}(n;\xi)}(1+nd(u,\xi))^k}.
\label{3.1}\end{equation}\end{lem}

The following lemma gives the extension of the  Schur inequality
on the interval $[-1,1]$ to the multidimensional case.

\begin{lem}(See \cite{J}, \cite[Lemma]{Sa}.) If  $P\in \Pi_{m-1}^d$  satisfies$$|P(x)|\le {\frac{C}{\sqrt{1-|x|^2}}},\ |x|<1,$$
then $$|P(x)|\le C m,\ \ x\in\Bbb B^d.$$
\end{lem}

The next lemma asserts that a nonnegative function on
$\mathbb{B}^d$ satisfying some inequality can be replaced by the
$p$-th power of a nonnegative polynomial on $\mathbb{B}^d$.

\begin{lem} Suppose that $\alpha>0,\ n\in\mathbb{N}_+$, and $f$ is a nonnegative function on $\mathbb{B}^d$ satisfying
\begin{equation}\label{3.2}f(x)\leq C(1+nd(x,y))^{\alpha}f(y),\ \ x, y\in\mathbb{B}^d. \end{equation}
Then for any $0<p<\infty$, there exists a nonnegative polynomial $h\in\Pi_n(\Bbb B^d)$ such that
\begin{equation}\label{3.3}
c^{-1}f(x)\le(h(x))^p\le c f(x),\ \ x\in\mathbb{B}^d,
\end{equation}
where $c>0$ depends only on $d,\ C,\ p$ and $\alpha$.

In addition, if $f$ is a nonnegative function on $\Bbb B^d$
satisfying \eqref{3.2} with $f(x)=f(-x),\ x\in \Bbb B^d$, then the
function $h$ in \eqref{3.3} can be chosen to satisfy $h(x)=h(-x),\
x\in\Bbb B^d$; if $f$ is a nonnegative function on $\Bbb B^d$
satisfying \eqref{3.2} with $f(x)=F(|x|),\ x\in \Bbb B^d$, where $F$ is
a nonnegative function on $[0,+\infty)$, then the function $h$ in
\eqref{3.3} can be chosen to satisfy $h(x)=G(|x|),\ x\in\Bbb B^d$, where
$G$ is a nonnegative function on $[0,+\infty)$.  \end{lem}

 The proof of Lemma 3.3 is based on the following two lemmas.
\begin{lem} Suppose that $\alpha>0,\ n\in\mathbb{N}_+$, and $f$ is a nonnegative function on $\Bbb S^{d}$ satisfying
\begin{equation}f(x)\leq C(1+n\rho(x,y))^{\alpha}f(y),\ \ x, y\in\mathbb{S}^{d}.\label{3.4}\end{equation}
Then for any $\ 0<p<\infty$ there exists a nonnegative spherical
polynomial $g\in\Pi_n(\Bbb S^d)$ such
that\begin{equation}\label{3.5}c^{-1}f(x)\leq (g(x))^p\leq c
f(x),\  \  x\in\mathbb{S}^{d},\end{equation}where $c>0$ depends
only on $d,\ C,\ p$ and $\alpha$.

In addition, suppose that $A\subset O(d+1)$, where $O(d+1)$ is the
group of all orthogonal transformations on $\Bbb R^{d+1}$.  If
$f(x)=f(\tau x)$ for any $\tau\in A$, then  the function $g$ in
\eqref{3.5} can be chosen to be a  polynomial satisfying $g(x)=g(\tau
x)$ for any $\tau\in A$.
\end{lem}

\begin{proof} Lemma 3.4 is essentially given in
\cite[Lemma 4.6]{D} or \cite[Lemma 5.4.4]{DaX}. Indeed, we set
$$g(x)=\int_{\Bbb
S^d}f(y)^{\frac{1}{p}}T_n(x\cdot y)d\sigma(y),\ x\in\Bbb
S^d,$$where $d\sigma$ is the usual normalized Lebesgue measure on
$\Bbb S^d$,  $$T_n(\cos \theta)=\gamma_n\Big(\frac{\sin
(n_1+\frac{1}{2})\theta}{\sin \frac{\theta}{2}}\Big)^{2m},\ \
m=\lfloor\frac{\alpha}{p}\rfloor+d+2,\
n_1=\lfloor\frac{n}{2m}\rfloor,$$ the notation $\lfloor x\rfloor$
stands for the largest integer not exceeding $x$, and $\gamma_n$
is chosen by
$$\int_0^\pi T_n(\cos \theta)\sin^{d-1}\theta d\theta=1.$$ It is
proved in \cite[Lemma 4.6]{D} that $g$ is a nonnegative spherical
polynomial $g\in\Pi_n(\Bbb S^d)$ and satisfies \eqref{3.5}.

If $f(x)= f(\tau x)$ for any $\tau\in A\subset O(d+1)$ and
$x\in\Bbb S^d$, then for $\tau\in A$ and $ x\in\Bbb S^d$, we have
\begin{align*}g(\tau x)&=\int_{\Bbb
S^d}f(y)^{\frac{1}{p}}T_n(\tau x\cdot y)d\sigma(y)=\int_{\Bbb
S^d}f(y)^{\frac{1}{p}}T_n(x\cdot \tau y)d\sigma(y)\\&=\int_{\Bbb
S^d}f(\tau y)^{\frac{1}{p}}T_n(x\cdot \tau y)d\sigma(
y)=\int_{\Bbb S^d}f(y)^{\frac{1}{p}}T_n( x\cdot y)d\sigma(y)=g(x),
\end{align*}which completes the proof of Lemma 3.4.
\end{proof}

We denote $$\tau_{d+1}\bar{x}:=(x,-x_{d+1}),\ \
\bar{x}=(x,x_{d+1})\in \Bbb R^{d+1}.$$Then $\tau_{d+1}\in O(d+1)$.

Throughout this section we define
$$\widetilde{\Pi_n(\Bbb S^{d})}:=\{p\in\Pi_n(\Bbb S^d)\ |\ p(\tau_{d+1}
\bar{x})=p(\bar{x}),\ {\rm for\ any}\ \bar{x}\in\Bbb S^d\}.$$

Consider the mapping $T$ from $\Pi_n(\Bbb B^d)$ to $\Pi_n(\Bbb S^d)$ defined by
$$T(f)(\bar{x}):=f(x),\ \bar{x}=(x,x_{d+1})\in\Bbb S^d.$$ We can show that this mapping has the following property.
\begin{lem}The mapping $T:\Pi_n(\Bbb B^d)\rightarrow\widetilde{\Pi_n(\Bbb S^{d})}$ is a bijection.
\end{lem}
\begin{proof}The proof is straightforward. \end{proof}

Next  we give the proof of Lemma 3.3.

\vskip 3mm

\noindent{\it Proof of Lemma 3.3.}

It follows from  $d(x,y)\leq \rho(\bar{x},\bar{y}),\ \
\bar{x}=(x,x_{d+1}),\ \bar{y}=(y, y_{d+1})\in\Bbb S^{d}$ and
\eqref{3.2} that $$T(f)(\bar{x})\leq
C(1+n\rho(\bar{x},\bar{y}))^{\alpha}T(f)(\bar{y}),\ \ \
\bar{x},\bar{y}\in\mathbb{S}^{d}. $$ Clearly, $T(f)(\tau_{d+1}
\bar x)=T(f)( \bar x)$,\ $\bar x=(x,x_{d+1})\in\Bbb S^d$.
According to Lemma 3.4, there exists a nonnegative spherical
polynomial $g\in\widetilde{\Pi_n(\Bbb S^{d})}$ such
that\begin{equation*}T(f)(\bar x)\asymp(g(\bar x))^p,\ \
\bar x\in\Bbb S^{d}.
\end{equation*}Hence, the function $h=T^{-1}g\in \Pi_n(\Bbb B^d)$
satisfies \eqref{3.3}.

We set $$\tilde \tau \bar x=-\bar x,\ \bar x=(x,x_{d+1})\in \Bbb
R^{d+1},$$ and $$O_{e_{d+1}}=\{\tau\in O(d+1)\ |\ \tau
e_{d+1}=e_{d+1}\},\ \ e_{d+1}=(0,\dots,0,1)\in \Bbb R^{d+1}.$$

Similarly, if  $f(x)=f(-x)$ for $x\in\Bbb B^d$, then $T(f)(\tau
\bar x)=T(f)(\bar x)$ for $\tau \in A=\{\tilde \tau, \tau_{d+1}\}$
and $\bar x\in \Bbb S^d$.  If $f(x)=F(|x|)$ for $x\in\Bbb B^d$,
then $T(f)(\tau \bar x)=T(f)(\bar x)$ for $\tau \in A=\{
\tau_{d+1}\}\,\cup\, O_{e_{d+1}}$ and $\bar x\in \Bbb S^d$.
According to Lemma 3.4, there exists a $g\in\widetilde{\Pi_n(\Bbb
S^{d})}$ satisfying \eqref{3.6} and $g(\tau \bar x)=g(\bar x)$ for
$\tau\in A$ and $\bar x\in \Bbb S^d$. Hence, the function
$h=T^{-1}g\in \Pi_n(\Bbb B^d)$ is just the function we need to search for.

Lemma 3.3 is proved. $\hfill\Box$

\vskip 3mm

Now we present the main result of this section which gives the
upper estimates of $\partial_{x_i} L_n(w_\mu;x, y), \ 1\le i\le d.$

\begin{thm}Fix  $\mu\geq0,\ n\in\mathbb{N}_+,\ 1\le i\le d, $ and  $k>0$. We have for all $ x,y\in\Bbb B^d$,
$$|\partial_{x_i} L_n(w_\mu;x,y)|\ll\min{(\frac{1}{\sqrt{1-| x|^2}},n)}\frac{n^{d+1}}{\sqrt{\mathcal{W_\mu}(n;x)}\sqrt{\mathcal{W_\mu}(n;y)}(1+nd(x,y))^k}.
$$
\end{thm}
\begin{rem}It can be seen easily that
$$\min{(\frac{1}{\sqrt{1-|x|^2}},n)}\asymp\frac{1}{(\mathcal{W_\mu}(n;x))^{\frac{1}{2\mu}}},\ \  x\in\Bbb B^d.$$

It follows from Theorem 3.6 that
$$|\partial_{x_i} L_n(w_\mu;x,
y)|\ll\frac{1}{(\mathcal{W_\mu}(n;x))^{\frac{1}{2\mu}+\frac{1}{2}}}\frac{n^{d+1}}{\sqrt{\mathcal{W_\mu}(n;y)}(1+nd(x,y))^k},\
x,y\in\Bbb B^d.
$$
\end{rem}

\begin{proof}
We first show that
\begin{equation}\label{3.6}|\partial_{x_i} L_n(w_\mu;x,y)|\ll\frac{\frac{n^{d+1}}{\sqrt{1-|
x|^2}}}{\sqrt{\mathcal{W_\mu}(n;x)}\sqrt{\mathcal{W_\mu}(n;y)}(1+nd(x,y))^k},\
\ |x|<1 .\end{equation} Let
$$e_1=(1,0,\dots,0),\ e_2=(0,1,\dots,0),\dots,\ e_d=(0,0,\dots,1)$$ be the orthonormal basis
in $\Bbb R^d$,  $|x|<1$, and $ y\in\mathbb{B}^d$. Then for
sufficient small real number $h$ and $1\le i\le d$, applying Lemma
3.1 we have
$$|L_n(w_\mu;x+he_i,y)-L_n(w_\mu;x, y)|\ll
\frac{n^{d+1}d(x+he_i,x)}{\sqrt{\mathcal{W_\mu}(n;x)}\sqrt{\mathcal{W_\mu}(n;y)}(1+nd(x,y))^k}.$$
By the equivalence of $d$ and $\tilde{d}$ we have
\begin{align*}&\quad\ d(x+he_i,x)\ll \tilde{d}(x+he_i,x) \\ &=\sqrt{|x+he_i-x|^2+(\sqrt{1-|x|^2}-\sqrt{1-|x+he_i|^2})^2}
\\&=\sqrt{h^2+(\sqrt{1-|x|^2}-\sqrt{1-|x|^2-2x_ih-h^2})^2},
\end{align*}where $x_i=x\cdot e_i$. We note that
\begin{align*}(\sqrt{1-|x|^2}-\sqrt{1-|x|^2-2x_ih-h^2})^2&=\Big(\frac{2x_ih+h^2}{\sqrt{1-|x|^2}+\sqrt{1-|x|^2-2x_ih-h^2}}\Big)^2\\&
\ll \frac{(2|h|+h^2)^2}{1-|x|^2}\ll \frac{h^2}{1-|x|^2},
\end{align*}
and $$h^2\leq \frac{h^2}{1-|x|^2},\ |x|<1.$$ It follows that
$$d(x+he_i,x)\ll \frac{|h|}{\sqrt{1-|x|^2}},$$ and
$$\Big|\frac{L_n(w_\mu;x+he_i, y)-L_n(w_\mu;x, y)}{h}\Big|\ll
 \frac{\frac{n^{d+1}}{\sqrt{1-|
 x|^2}}}{\sqrt{\mathcal{W_\mu}(n;x)}\sqrt{\mathcal{W_\mu}(n;y)}(1+nd(x,y))^k}.$$
By the definition of $\partial_{x_i} L_n(w_\mu;x,y)$ and the above
inequality, we obtain \eqref{3.6}.

Next, we complete the proof by showing that
\begin{equation}\label{3.7}
  |\partial_{x_i} L_n(w_\mu;x, y)|\ll\frac{n^{d+2}}{\sqrt{\mathcal{W_\mu}(n;x)}\sqrt{\mathcal{W_\mu}(n;y)}(1+nd(x,y))^k},\ \ x,y\in\mathbb{B}^d.
\end{equation}

In fact, for fixed $y\in \Bbb B^d$, we set
$$F(x):=\partial_{x_i}L_n(w_\mu;x,y),\ G(x):=\sqrt{\mathcal{W_\mu}(n;x)}(1+nd(x,y))^k .$$
We know that for a fixed $y\in \mathbb{B}^{d}$,
$L_n(w_\mu;x,y)\in\Pi_{2n}^d,$ so $F(x)\in\Pi_{2n-1}^d.$
It follows from \cite[Equation (4.23)]{PX} that
\begin{equation*}
\frac
{1}{2^\mu(1+nd(x,y))^{2\mu}}\leq\frac{\mathcal{W_\mu}(n;x)}{\mathcal{W_\mu}(n;y)}\leq{2^\mu(1+nd(x,y))^{2\mu}},\
x, y\in\mathbb{B}^{d}.
\end{equation*}
By the  inequality
\begin{equation}\label{3.8}
1+nd(x,y)\le
(1+nd(x,z))(1+nd(y,z)),\ x,y,z\in \Bbb B^d,
\end{equation}
we have\begin{align}\label{3.9}G(x)&=\sqrt{\mathcal{W_\mu}(n;x)}(1+nd(x,y))^k\notag\\
&\ll\sqrt{\mathcal{W_\mu}(n;z)}(1+nd(x,z))^{\mu} (1+nd(x,z))^k
(1+nd(z,y))^k
\nonumber\\&\ll(1+nd(x,z))^{\mu+k}\sqrt{\mathcal{W_\mu}(n;z)}(1+nd(y,z))^k\nonumber\\
&=(1+nd(x,z))^{\mu+k}G(z).\end{align} According to Lemma 3.3,
there exists a nonnegative polynomial $G_n\in\Pi_n^d$ such that
$$G(x)\asymp G_n(x),\ \ {\rm for\ all}\ x\in\Bbb B^d.$$
By \eqref{3.6} we have
$$n^{-d-1}|\sqrt{\mathcal{W_\mu}(n;y)}F(x)G_n(x)|\ll\frac{1}{\sqrt{1-|x|^2}},\ |x|<1.$$ Note that  for fixed $y\in \Bbb B^d$,
 $\sqrt{\mathcal{W_\mu}(n;y)}F(\cdot)G_n(\cdot)\in \Pi_{3n-1}^d$. Applying Lemma 3.2 we obtain  $$n^{-d-1}|\sqrt{\mathcal{W_\mu}(n;y)}F(x)G_n(x)|\ll n,\ x\in \Bbb B^d,$$
which leads that
$$n^{-d-1}|\sqrt{\mathcal{W_\mu}(n;y)}F(x)G(x)|\ll n,\ x\in \Bbb
B^d.$$The  desired result \eqref{3.7} is proved. This completes
the proof of Theorem 3.6.
\end{proof}

\begin{cor}
For all $ x, y\in\mathbb{B}^d,\ 1\le i\le d,$ we have
$$\|\partial_{x_i} L_n(w_\mu;\cdot,y)\|_{1,\mu}\ll n^2,$$and
$$\|\partial_{x_i} L_n(w_\mu;x,\cdot)\|_{1,\mu}\ll n^2.$$
\end{cor}
\begin{proof}
Corollary 3.8  follows from  Theorem 3.6 and Lemma 2.1
immediately.
\end{proof}

\section{The Christoffel function on the ball}

\

For $1\le p<\infty$, we define the Christoffel function on $\Bbb B^d$ by
\begin{equation}\label{4.1}\Lambda_{n,p}(x):=\min_{P(x)=1,P\in\Pi_n^d}\| P\|_{p,w}^p, \ \ x\in\mathbb{B}^d,\end{equation}where the minimum is taken for all polynomials in $d$ variables of total degrees at most $n$ that take the value 1 at $x$.
In particular, when $p=2$, it follows from \cite[Theorem 3.5.6]{DX} that
\begin{align}\Lambda_{n,2}(x)=\Big(\sum_{j=1}^{N_n^d}|P_j(x)|^2\Big)^{-1},\ \ x\in\Bbb B^d,\label{4.2}\end{align}
where $\{P_1,P_2,\dots,P_{N_n^d}\}$ is an orthonormal basis for $\Pi_n^d$ in $L_{2,w}$.

The main result of this section is the following  theorem which
gives the estimates for the Christoffel functions
$\Lambda_{n,p}(x)$.
\begin{thm} Let $n\in\mathbb{N_+}$,  $1\le p<\infty$,  and $w$ be a doubling weight on $\Bbb B^d$. Then we have \begin{equation}\Lambda_{n,p}(x)\asymp w({\bf B}(x,\frac{1}{n})),\ x\in\Bbb B^d.
\label{4.3}\end{equation}
\end{thm}

\begin{rem}  The upper estimates of $\Lambda_{n,2}(x)$ is given  in
\cite[Propositions 2.17]{Xu2}. The lower estimates of
$\Lambda_{n,2}(x)$ for the Jacobi weight $w_\mu$ is given  in
\cite[Propositions 2.18]{Xu2}. It follows from \eqref{4.2} and
\eqref{4.3} that
\begin{equation}\label{4.4}
\sum_{j=1}^{N_n^d}|P_j(x)|^2 \asymp\frac 1{w({\bf
B}(x,\frac{1}{n}))}, \ \ x\in\Bbb B^d.\end{equation}
\end{rem}

In order to obtain \eqref{4.3}, we need to establish weighted
polynomial inequalities on $\Bbb B^d$ with doubling weights (See
\cite{D, Xu2} for more details). Similar weighted polynomial
inequalities on the sphere with doubling weights were established
in \cite{D} or \cite[Chapter 5]{DaX}.
 In what follows, we shall use the symbol $s_w$ to
denote a positive number  such that
\begin{equation}\label{4.5}
  \sup\limits_{x\in\Bbb B^d,r>0}\frac{w(\mathbf{B}(x,2^mr))}{w(\mathbf{B}(x,r))}\le C_{L_w}2^{ms_w},\ \ \ m=0,1,2,\dots,
\end{equation}
where $C_{L_w}$ is a constant depending only on the doubling
constant $L_w$ of $w$. Clearly, such a constant $s_w$ exists, for
example, we can take $s_w=\frac{\log L_w}{\log2}$. It follows from
\eqref{4.5} that
\begin{equation}\label{4.6}
  w(\mathbf{B}(x,t))\le 2^{s_w}C_{L_{w}}(\frac
  tr)^{s_{w}}w(\mathbf{B}(x,r)),\ \ x\in \Bbb B^d, \ 0<r<t,
\end{equation}
and
\begin{equation}\label{4.7}
  w(\mathbf{B}(x,\frac1n))\le 2^{s_w}C_{L_{w}}(1+nd(x,y))^{s_{w}}w(\mathbf{B}(y,\frac1n)),\ \ x,y\in\Bbb
  B^d, \ n\in\mathbb{N}_+.
\end{equation}

Similar to the one-dimensional  case (see \cite{MT}), for
$n\in\Bbb N_+$ and a doubling weight $w$ on $\Bbb B^d$  we define
\begin{equation}\label{4.8}
  w_n(x):=\frac{1}{w_{1/2}({\bf B}(x,\frac{1}{n}))}w({\bf B}(x,\frac{1}{n}))\asymp\frac{n^{d}}{\mathcal{W}_{1/2}(n;x)}w({\bf B}(x,\frac{1}{n})).
\end{equation}
We can show that $w_n$ is also a doubling weight with the doubling
constant $L_{w_n}$ depending only on $d$ and the doubling constant
$L_w$ of $w$. It follows from \cite[Equation (4.23)]{PX} that
\begin{equation}\label{4.9}
 \frac{1}{\sqrt{2}(1+nd(x,y))}\le
 \frac{\mathcal{W}_{1/2}(n;x)}{\mathcal{W}_{1/2}(n;x)}\le\sqrt{2}(1+nd(x,y)).
\end{equation}
By \eqref{4.7}, \eqref{4.8}, and \eqref{4.9} we have
\begin{equation}\label {4.10}
  w_n(x)\ll (1+nd(x,y))^{s_{w}+1}w_n(y),\ \ x,y\in\Bbb B^d.
\end{equation}

\vskip 3mm

Let $\beta>0$, $f\in C(\Bbb B^d)$, and $n\in\Bbb N_+$. We denote
$$f_{\beta,n}^*(x)=\max_{y\in\Bbb B^d}|f(y)|(1+nd(x,y))^{-\beta},\ x\in\mathbb{B}^d. $$ It is immediate to infer that
\begin{equation}\label{4.11}|f(x)|\leq f_{\beta,n}^*(x)\leq(1+nd(x,y))^{\beta}f_{\beta,n}^*(y),\ x,y\in\Bbb B^d.\end{equation}

\vskip 3mm

\begin{lem}(See \cite[Corollary 2.12]{Xu2}.) For $n\in\mathbb{N}_+$, $f\in \Pi_n^d$, $0< p\le\infty$ and $\beta>s_w/p$, we have for all
$$\| f\|_{p,w}\le \| f_{\beta,n}^*\| _{p,w}\le c\| f\| _{p,w},$$
where $c$ depends also on $L_w$ and $\beta$ when $\beta$ is either large or close to $s_w/p$.
\end{lem}

\vskip 3mm

 For $\varepsilon>0$, a subset $\Lz$ of $\Bbb B^d$
is called $\varepsilon$-separated if $d(x,y)\ge \varepsilon$ for
any two distinct points $x,y\in\Lz$. Furthermore, if an
$\varepsilon$-separated subset $\Lz$ of $\Bbb B^d$ satisfies
$\Bbb
B^d=\bigcup\limits_{\oz\in\Lz}\mathbf{B}(\oz,\varepsilon)$,
then we call it a maximal $\varepsilon$-separated subset of $\Bbb
B^d$. It is easy to clarify that for a maximal $\varepsilon$-separated subset $\Lz\subset \Bbb B^d$, there exists a constant $C_d$ such that
\begin{equation}\label{4.12}
  1\le \sum_{\oz\in\Lz}\chi_{\mathbf{B}(\oz,\varepsilon)}(x)\le C_d,\ {\rm for\ any}\ x\in\Bbb B^d.
\end{equation}

\begin{cor}Let  $n\in\mathbb{N}_+$ and $0< p<\infty$. Then for
all  $f\in \Pi_n^d$  we have
$$\| f\|_{p,w_n}\asymp \| f\| _{p,w}.$$
\end{cor}
\begin{proof}
Since  $w_n$ is a doubling weight with a doubling constant
depending only on $d$ and that of $w$,  by Lemma 3.4   it is
sufficient to prove that
\begin{equation}\label{4.13}
  \|f_{2s/p,n}^\ast\|_{p,w_n}\asymp\|f_{2s/p,n}^\ast\|_{p,w},\ {\rm with}\ s=\max\{s_w,s_{w_n}\}.
\end{equation}
To obtain \eqref{4.13}, we let $\Lambda\subset\Bbb B^d$ be a
maximal $\frac1n$-separable subset.

Note that for $\oz\in\Lz$ and $x\in {\bf B}(\oz,\frac1n)$,
$$f_{2s/p,w}^\ast(x)\asymp f_{2s/p,w}^\ast(\oz)\ \ {\rm and}\ \ w_n(x)\asymp w_n(\oz).$$
Combining with \eqref{4.12} and \eqref{4.8}, we obtain
\begin{align*}
  \|f_{2s/p,n}^\ast\|_{p,w}^p&\asymp\sum\limits_{\oz\in\Lambda}\int_{{\bf B}(\oz,\frac1n)}(f_{2s/p,n}^\ast(x))^pw(x)dx
  \\&\asymp \sum\limits_{\oz\in\Lambda}(f_{2s/p,n}^\ast(\oz))^pw({\bf B}(\oz,\frac{1}{n}))
  \\&\asymp\sum\limits_{\oz\in\Lambda}\int_{{\bf B}(\oz,\frac{1}{n})}(f_{2s/p,n}^\ast(\oz))^pw_n(\oz)dx
  \\&\asymp\sum\limits_{\oz\in\Lambda}\int_{{\bf B}(\oz,\frac1n)}(f_{2s/p,n}^\ast(x))^pw_n(x)dx\\ &\asymp \|f_{2s/p,n}^\ast\|_{p,w_n}^p,
\end{align*}proving \eqref{4.13}. This completes the proof of Corollary 4.4.
\end{proof}

Now we give the proof of Theorem 4.1.

\vskip 3mm

\noindent{\it Proof of Theorem 4.1.}

We begin by proving the upper estimate of $\Lambda_{n,p}(x)$. For
a fixed $x\in \Bbb B^d$ and $k>s_w+d$, we set
$$P_x(y):=\frac{1}{(1+nd(x,y))^{k/p}},\ \
y\in\Bbb B^d.$$ Then $P_x$ is a nonnegative function on $\Bbb B^d$
with $P_x(x)=1$ and satisfies \eqref{3.2}. According to Lemma
3.6, there exists a nonnegative polynomial $P_n\in\Pi_n^d$  such
that
$$ P_x(y)\asymp P_n(y),\ \ y\in\Bbb B^d.$$
This leads that $P_n(x)\asymp 1$. By the definition of
$\Lz_{n,p}(x)$, Lemma 4.3,  \eqref{4.8}, and Lemma 2.1, we obtain
that
\begin{align*}
  &\qquad \Lambda_{n,p}(x)\ll \|P_n\|_{p,w}^p\asymp\|P_n\|_{p,w_n}^p\asymp\|P_x\|_{p,w_n}^p
 \\& =\int_{\Bbb B^d}\frac{1}{(1+nd(x,y))^{k}}w_n(y)dy\asymp \int_{\Bbb B^d}\frac{1}{(1+nd(x,y))^{k}} \frac {w({\bf B}(y,\frac{1}{n}))}{ n^{-d}\mathcal{W}_{1/2}(n;y)}dy
  \\&\ll w({\bf B}(x,\frac{1}{n})) \int_{\Bbb B^d}\frac{n^{d}}{(1+nd(x,y))^{k-s_w}\mathcal{W}_{1/2}(n;y)} dy \ll w({\bf B}(x,\frac{1}{n})).
\end{align*}

We now turn to show the lower bound of $\Lambda_{n,p}(x)$.
Let $P\in\Pi_n^d$ with $P(x)=1$. It follows from
\eqref{4.11} that for any $y\in \mathbf{B}(x,\frac{1}{n})$,
\begin{equation}P_{2s_w/p,n}^*(y)\asymp  P_{2s_w/p,n}^*(x)\geq P(x)=1.\label{4.14}\end{equation} Using Lemma 3.3 we obtain that
\begin{align*}
\Lambda_{n,p}(x)&=\min_{P(x)=1,P\in\Pi_n^d}\|P\|_{p,w}^p\asymp \min_{P(x)=1,P\in\Pi_n^d}\|P_{2s_w/p,n}^*\|_{p,w}^p
\\&\gg \min_{P_{2s_w/p,n}^*(x)\ge 1,P\in\Pi_n^d}\int_{\mathbf{B}(x,\frac{1}{n})} |P_{2s_{w}/p,n}^*(y)|^p
w(y)dy\\ &\gg \int_{\mathbf{B}(x,\frac{1}{n})}w(y)dy
=w(\mathbf{B}(x,\frac{1}{n})).
\end{align*}
This shows the desired lower estimate of  $\Lambda_{n,p}(x)$.

  Theorem 4.1 is proved. $\hfill\Box$

\section {Proofs of Theorems 1.1 and  1.2}

\

\noindent{\it Proof of Theorem 1.1.}

Let $P(x)\in\Pi_n^d$. By \eqref{2.3}, the H\"older inequality, and Corollary 3.8, we have for $1\le p<\infty,\ 1/p+1/p'=1$,
\begin{align}
|\partial_i P(x)|&\leq \int_{\mathbb{B}^{d}}|P(y)|
|\partial_{x_i}L_n(w_{1/2};x,y)|
dy\notag\\&\leq\Big(\int_{\mathbb{B}^{d}}|P(y)|^p|\partial_{x_i}L_n(w_{1/2};x,y)|dy\Big)^{\frac{1}{p}}
 \Big(\int_{\mathbb{B}^{d}}|\partial_{x_i} L_n(w_{1/2};x,y)|dy\Big)^{\frac{1}{p^\prime}}
\notag\\&\ll\Big(\int_{\mathbb{B}^d}|P(y)|^p|\partial_{x_i}L_n(w_{1/2};x,y)|dy\Big)^{\frac{1}{p}}n^\frac{2}{p\prime}.\label{5.1}
\end{align}
By Corollary 4.4,  \eqref{5.1}, Fubini's Theorem, \eqref{3.7},
\eqref{4.9} and \eqref{4.10}, and Lemma 2.1,  we obtain for
$k>s_{w}+d+2$,
\begin{align*}\|\partial_iP\|_{p,w}^p&\ll  \|\partial_iP\|_{p,w_n}^p\ll n^\frac{2p}{p\prime}\int_{\mathbb{B}^d}
\Big(\int_{\mathbb{B}^{d}} |P(y)|^p|\partial_{x_i}
L_n(w_{1/2};x,y)|dy\Big)w_n(x)dx\\ &=
n^\frac{2p}{p\prime}\int_{\mathbb{B}^d}
|P(y)|^p\Big(\int_{\mathbb{B}^{d}} |\partial_{x_i} L_n(w_{1/2};x,y)|
w_n(x)dx\Big) dy
\\&\ll n^{\frac{2p}{p\prime}}\int_{\mathbb{B}^{d}}|P(y)|^p\Big(\int_{\mathbb{B}^d}\frac{n^{d+2}w_n(x)}{\sqrt{\mathcal{W}_{1/2}(n;x)\mathcal{W}_{1/2}(n;y)}(1+nd(x,y))^k}dx\Big)dy
\\&\ll n^{\frac{2p}{p\prime}+2}\int_{\mathbb{B}^{d}}|P(y)|^p
\Big(\int_{\mathbb{B}^d}\frac{n^{d}(1+nd(x,y))^{s_{w}+1}w_n(y)
}{\mathcal{W}_{1/2}(n;x)(1+nd(x,y))^{k-1}} dx\Big)dy
\\&= n^{\frac{2p}{p\prime}+2}\int_{\mathbb{B}^{d}}|P(y)|^p
\Big(\int_{\mathbb{B}^d}\frac{n^{d}}{\mathcal{W}_{1/2}(n;x)(1+nd(x,y))^{k-s_{w}-2}}
dx\Big)w_n(y)dy
\\&\ll\|P\|_{p,w_n}^pn^{2+\frac{2p}{p\prime}}\ll\|P\|_{p,w}^pn^{2+\frac{2p}{p\prime}},
\end{align*} which yields that
$$\|\partial_i P\|_{p,w}\ll n^2\|P\|_{p,w}.$$
It follows that \begin{align*}\|\,|\nabla P|\,\|_{p,w}^p&=
\int_{\mathbb{B}^d}\big(\sum\limits_{i=1}^d|\partial_iP(y)|^2\big)^{\frac{p}{2}}w(y)dy\\
&\asymp
\sum\limits_{i=1}^d\int_{\mathbb{B}^d}|\partial_iP(y)|^pw(y)dy\ll
n^{2p}\|P\|_{p,w}^p.\end{align*}

Now we show that for the Jacobi weight $w_\mu$, the exponent $2$ is sharp. For $1\le i\le d$,
let $f\in \Pi_n^1$, and
$$\bar{P}(x)=\bar P (x_1,\cdots,x_i,\cdots,x_d)=f(x_i)\in
\Pi_n^d.$$
 By \cite[Corollary A.5.3]{DaX} we have
$$\|\partial_i\bar{P}\|_{p,w_\mu}^p=\int_{\mathbb{B}^{d}}|f^\prime(x_i)|^pw_\mu(x)dx=c_{\mu,d}\int_{-1}^{1}|f^\prime(x_i)|^pw_\lambda(x_i)dx_i,$$
and
$$\|\bar{P}\|_{p,w_\mu}^p=\int_{\mathbb{B}^{d}}|f(x_i)|^pw_\mu(x)dx=c_{\mu,d}\int_{-1}^{1}|f(x_i)|^pw_\lambda(x_i)dx_i,$$
where $w_\lz(t)=(1-t^2)^{\lz-1/2}, \ t\in[-1,1]$,
$\lambda=\mu+\frac{d-1}{2}$, and $c_{\mu,d}$ is a positive
constant  depending only on $\mu$ and $d.$

Since $\partial_j\bar{P}=0$, for $j\neq i,\ 1\le
j\le d$, it follows that
$$\frac{\|\,|\nabla\bar{P}|\,\|_{p,w_\mu}}{\|\bar{P}\|_{p,w_\mu}}=\frac{\|\partial_i\bar{P}\|_{p,w_\mu}}{\|\bar{P}\|_{p,w_\mu}}=\frac{\|f^\prime\|_{p,w_\lambda}}{\|f\|_{p,w_\lambda}}.$$
Combining with \eqref{1.4}, we get
$$\sup\limits_{0\neq P\in\Pi_n^d}\frac{\|\,|\nabla P|\,\|_{p,w_\mu}}{\|P\|_{p,w_\mu}}
\gg \sup\limits_{0\neq
f\in\Pi_n^1}\frac{\|f^\prime\|_{p,w_\lambda}}{\|f\|_{p,w_\lambda}}\asymp
n^2.$$

Theorem 1.1 is proved.
$\hfill\Box$

\vskip 3mm
\noindent{\it Proof of Theorem 1.2.}

According to \eqref{2.4}, \eqref{2.5}, and Corollary 4.4, it
suffices to show that for $1\le i\le d$,
\begin{equation}\sum_{j=1}^{N_n^d} \|\partial_i
P_j\|_{2,w_n}^2=\int_{\mathbb{B}^d}\Big(\sum_{j=1}^{N_n^d}|\partial_i
P_j(x)|^2\Big)w_n(x)dx\ll n^{d+3}.\label{5.2}\end{equation}

Choose  $k>s_w+d+1$. It follows from \eqref{2.3}, Remark 3.7, the
H\"{o}lder inequality, and Lemma 2.1 that
\begin{align*}|\partial_i P_j(x)|^2&\leq \Big(\int_{\mathbb{B}^{d}}|P_j(y)| |\partial_{x_i} L_n (w_{1/2};x,y)|dy\Big)^2\\&\ll
 \Big(\int_{\mathbb{B}^{d}}|P_j(y)|\frac{1}{(\sqrt{\mathcal{W}_{1/2}(n;x)})^3}\frac{n^{d+1}}{\sqrt{\mathcal{W}_{1/2}(n;y)}(1+nd(x,y))^k}dy\Big)^2
\\&\ll\frac{n^{d+2}}{(\mathcal{W}_{1/2}(n;x))^3}\int_{\mathbb{B}^{d}}\frac{|P_j(y)|^2}{(1+nd(x,y))^k}dy
\int_{\mathbb{B}^{d}}\frac{n^{d}}{\mathcal{W}_{1/2}(n;y)(1+nd(x,y))^k}dy
\\&\ll\frac{n^{d+2}}{(\mathcal{W}_{1/2}(n;x))^3}\int_{\mathbb{B}^{d}}\frac{|P_j(y)|^2}{(1+nd(x,y))^k}dy.
\end{align*}
By \eqref{4.4}, \eqref{4.8},  and Lemma 2.1,  we get that
\begin{align*}
\sum_{j=1}^{N_n^d}|\partial_iP_j(x)|^2
&\ll\frac{n^{d+2}}{(\mathcal{W}_{1/2}(n;x))^3}\int_{\mathbb{B}^{d}}\frac{\sum\limits_{j=1}^{N_n^d}|P_j(y)|^2}{(1+nd(x,y))^k}dy
\\&\ll\frac{n^{d+2}}{(\mathcal{W}_{1/2}(n;x))^3}\int_{\Bbb B^d}\frac{1}{w({\bf B}(y,\frac1n))(1+nd(x,y))^k}dy
\\&\ll\frac{n^{d+2}}{(\mathcal{W}_{1/2}(n;x))^3}\int_{\Bbb B^d}\frac{1}{n^{-d} w_n(y)\mathcal{W}_{1/2}(n;y)(1+nd(x,y))^k}dy
\\&\ll\frac{n^{d+2}}{(\mathcal{W}_{1/2}(n;x))^3 w_n(x)}\int_{\Bbb B^d}\frac{n^{d}}{\mathcal{W}_{1/2}(n;y)(1+nd(x,y))^{k-s_w-1}}dy
\\&\ll\frac{n^{d+2}}{(\mathcal{W}_{1/2}(n;x))^3w_n(x)}.
\end{align*}
Then, according to \cite[Lemma 4]{WZ} we have
$$
\int_{\mathbb{B}^{d}}\frac{1}{(\mathcal{W}_{1/2}(n;x))^3}dx\asymp n  ,$$
which leads that
\begin{align*}
\int_{\mathbb{B}^d}\Big(\sum_{j=1}^{N_n^d}|\partial_iP_j(x)|^2\Big)w_n(x)dx
&\ll\int_{\mathbb{B}^d}\frac{n^{d+2}}{(\mathcal{W}_{1/2}(n;x))^3}dx\ll
n^{d+3},\end{align*}proving \eqref{5.2}. Theorem 1.2 is proved.
$\hfill\Box$

\vskip 3mm

\noindent{\bf Acknowledgment}  Jiansong Li and Heping Wang  were
supported by the National Natural Science Foundation of China
(Project no. 11671271) and
 the  Natural Science Foundation of Beijing Municipality
 (1172004). Kai Wang was supported by the National Natural Science Foundation of China
(Project no. 11801245) and
 the  Natural Science Foundation of Hebei Province (Project no.
 A2018408044).

\end{document}